\documentclass[10pt,a4paper]{article}
\usepackage{amsmath}
\usepackage{amssymb}
\usepackage{amsthm}
\usepackage{amsfonts}
\usepackage{enumerate}
\usepackage{pstricks}
\usepackage{pst-plot}
\usepackage{pst-poly}
\usepackage{graphics} %cx na konci pro PDF
\usepackage{graphicx}
\usepackage{url}
\usepackage{booktabs}
\usepackage{verbatim}
\usepackage[margin=1.3in]{geometry}
\newcommand{\tluste}[1]{\mbox{\mathversion{bold}$ #1 $}}
\newcommand{\vr}[1]{{{#1}}}

\newcommand{\mace}[1]{{{#1}}}
\newcommand{\mna}[1]{{\mathcal{#1}}}

\newcommand{\omace}[1]{\mbox{$\overline{\mace{#1}}$}} 
\newcommand{\umace}[1]{\mbox{$\underline{\mace{#1}}$}} 
\newcommand{\imace}[1]{\mbox{$\tluste{#1}$}}

\def\Mid#1{{#1^c}}
\def\Rad#1{{#1^\Delta}}
\def\Mag#1{\mathop{\mathrm{mag}}#1}	%magnitude of interval
\def\comp#1{{\langle{#1}\rangle}}%comparison matrix
\newcommand{\ovr}[1]{\mbox{$\overline{\vr{#1}}$}} 
\newcommand{\uvr}[1]{\mbox{$\underline{\vr{#1}}$}}

\newcommand{\onum}[1]{\mbox{$\overline{{#1}}$}} 
\newcommand{\unum}[1]{\mbox{$\underline{{#1}}$}}

\newcommand{\ivr}[1]{\mbox{$\tluste{#1}$}} 

\newcommand{\inum}[1]{\mbox{$\tluste{#1}$}}

\newcommand{\R}[0]{{\mathbb{R}}}

\newcommand{\IR}[0]{{\mathbb{IR}}}

\newcommand{\Ss}[0]{\mbox{\large$\Sigma$}}

\def\ihull{{\mbox{$\square$}}}

\newcommand{\mmid}[0]{;\,}		%pouzivejte v definici mnozin!

\newcommand{\seznam}[1]{{\{1, \ldots, {#1}\}}}

\DeclareMathOperator{\diag}{diag}	%diagonal matrix

\def\nref#1{$(\ref{#1})$}

\newtheorem{theorem}{Theorem}

\newtheorem{proposition}{Proposition}

\theoremstyle{definition}
\newtheorem{alg}{Algorithm}
 
\newtheorem{example}{Example}
\begin{document}

\title{A New Operator and Method for Solving Interval Linear Equations}

%\author{Milan Hlad\'{i}k \and Evgenija D. Popova}
\author{
  Milan Hlad\'{i}k\footnote{
Charles University, Faculty  of  Mathematics  and  Physics,
Department of Applied Mathematics, 
Malostransk\'e n\'am.~25, 11800, Prague, Czech Republic, 
e-mail: \texttt{milan.hladik@matfyz.cz}
}
%\footnote{University of Economics, Faculty of Informatics and Statistics,
%n\'{a}m. W. Churchilla 4, 13067, Prague, Czech Republic
%}
%\and Jaroslav Hor\'{a}\v{c}ek\footnote{
%Charles University, Faculty  of  Mathematics  and  Physics,
%Department of Applied Mathematics, 
%Malostransk\'e n\'am.~25, 11800, Prague, Czech Republic, 
%e-mail: \texttt{horacek@kam.mff.cuni.cz}}
}

\date{\today}
\maketitle

\begin{abstract}%\normalsize
We deal with interval linear systems of equations. We present a new operator, which generalizes the interval Gauss--Seidel method. Also, based on the new operator and properties of the well-known methods, we propose a new algorithm, called the magnitude method. We illustrate by numerical examples that our approach overcomes some classical methods with respect to both time and sharpness of enclosures.

\end{abstract}

%\textbf{Keywords:}\textit{ Interval matrix, interval analysis, eigenvalue, eigenvalue bounds.}

%%%%%%%%%%%%%%%%%%%%%%%%%%%%%%%%%%%%%%%%%%%%%%%%%%%%%%%%%%%%%%% 
% INTRODUCTION
%%%%%%%%%%%%%%%%%%%%%%%%%%%%%%%%%%%%%%%%%%%%%%%%%%%%%%%%%%%%%%% 
\section{Introduction}

We consider a system of linear equation with coefficients varying inside given intervals, and we want to find a guaranteed enclosure for all emerging solutions.
Since determining the best enclosure to the solution set is an NP-hard problem \cite{Fie2006}, the approaches to calculate it may be computationally expensive \cite{Jan1997,Roh1989,Sha1995} in the worst case. 
That is why the research was driven to develop cheap methods for enclosing the solution set, not necessarily optimally.
There are meny methods known; see e.g.\ \cite{Bea1998,Fie2006,HanSen1981,Han1992, MooKea2009,Neu1990,Neu1999,RohRex1998,Rum2010}.
Extensions to parametric interval systems were studied in \cite{Hla2012d,PopHla2013,Rum2010}, among others, and quantified solutions were investigated e.g.\ in \cite{Pop2012,PopHla2013,Sha2002}.

We will use the following interval notation.
An interval matrix $\imace{A}$ is defined as
$$
\imace{A}:=[\umace{A},\omace{A}]=\{A\in\R^{m\times n}\mmid \umace{A}\leq A\leq\omace{A}\},
$$
where $\umace{A},\omace{A}\in\R^{m\times n}$ are given. The center and radius of $\imace{A}$ are respectively defined as
$$
\Mid{A}:=\frac{1}{2}(\umace{A}+\omace{A}),\quad
\Rad{A}:=\frac{1}{2}(\omace{A}-\umace{A}).
$$
The set of all $m$-by-$n$ interval matrices is denoted by $\IR^{m\times n}$. Interval vectors and intervals can be regarded as special interval matrices of sizes $m$-by-$1$ and $1$-by-$1$, respectively.
For a definition of interval arithmetic see e.g. \cite{MooKea2009,Neu1990}.
Extended interval arithmetic with improper intervals of type $[\onum{a},\unum{a}]$, $\onum{a}>\unum{a}$, was discussed e.g.\ in  \cite{Kau1980,Sha2002}. We will use improper intervals only for the simplicity of exposition of interval expressions. For example, $\inum{a}+[b,-b]$, where $b>0$, is a shortage for the interval $[\unum{a}+b,\onum{a}-b]$.

%For $\mna{S}\subseteq\R^n$, the symbol $\ihull\mna{S}$ denotes the interval hull of $\mna{S}$, i.e., the smallest interval vector containing$ \mna{S}$.
 The magnitude of an $\imace{A}\in\IR^{m\times n}$ is defined as $\Mag(\imace{A}):=\max(|\umace{A}|,|\omace{A}|)$, where $\max(\cdot)$ is understood entrywise. The comparison matrix of $\imace{A}\in\IR^{n\times n}$ is the matrix $\comp{A}\in\R^{n\times n}$ with entries
\begin{align*}
%\comp{A}_{ii}&:=\min_{a\in\inum{a}_{ii}}|a|,\quad i=1,\dots,n,\\
\comp{A}_{ii}&:=\min\{|a|\mmid a\in\inum{a}_{ii}\},\quad i=1,\dots,n,\\
\comp{A}_{ij}&:=-\Mag(\inum{a}_{ij}),\quad i\not=j.
\end{align*}

Consider a system of interval linear equations
$$
\imace{A}x=\ivr{b}
$$
where $\imace{A}\in\IR^{n\times n}$ and $\ivr{b}\in\IR^n$. The corresponding solution set is defined as
$$
\Ss:=\{x\in\R^n\mmid \exists A\in\imace{A}\exists b\in\ivr{b}: Ax=b\}.
$$
The aim is to compute as tight as possible enclosure of $\Ss$, that is, an interval vector $\ivr{x}\in\IR^n$ containing $\Ss$. By  $\ivr{\Ss}:=\ihull\Ss$ we denote the interval hull of $\Ss$, i.e., the smallest interval enclosure of $\Ss$. Thus, enclosing $\Ss$ or $\ivr{\Ss}$ is the same objective.

Throughout the paper, we assume that $\Mid{A}=I_n$, that is, the midpoint of $\imace{A}$ is the identity matrix. This assumption is not without loss of generality, but most of the solvers utilize preconditioning
$$
(R\imace{A})x=R\ivr{b},
$$
where $R$ is the numerically computed inverse of $\Mid{A}$. Thus, the midpoint of $R\imace{A}$ is nearly the identity matrix. To be numerically save, we relax the system to
$$
[I_n-\Mag(I_n-R\imace{A}),I_n+\Mag(I_n-R\imace{A})]x=R\ivr{b}.
$$
Even though preconditioning causes an expansion of the solution set, it is more easy to handle. Since we do not miss any old solution, any enclosure to the preconditioned system is a valid enclosure for the original one as well.

The assumption $\Mid{A}=I_n$ has many consequences. 
The solution set of such interval linear system is bounded if and only if $\rho(\Rad{A})<1$, where $\rho(\Rad{A})$ stands for the spectral radius of $\Rad{A}$. So in the rest of the paper we assume that this is satisfied.

Another nice property of the interval system in question is that the interval hull of the solution set can be determined exactly (up to the numerical accuracy) by calling the Hansen--Bliek--Rohn method \cite{Fie2006,Roh1993}. Ning and Kearfott \cite{Neu1999,NinKea1997} proposed an alternative formula to compute $\inum{\Ss}$. We state it below and use the following notation
\begin{align*}
u&:=\comp{A}^{-1}\Mag(\ivr{b}),\\
d_i&:=(\comp{A}^{-1})_{ii},\quad i=1,\dots,n,\\
\alpha_i&:=\comp{a_{ii}}-1/d_i,\quad i=1,\dots,n.
\end{align*}
Notice also that the comparison matrix $\comp{A}$ can now be expressed as $\comp{A}=I_n-\Rad{A}$.

\begin{theorem}[Ning--Kearfott, 1997]\label{thmNK}
We have
\begin{align}\label{inclNK}
\inum{\Ss}_i=
% \frac{\ivr{b}_i-[-u_i/d_i+\Mag(\ivr{b}_i),u_i/d_i-\Mag(\ivr{b}_i)]}
 \frac{\ivr{b}_i+(u_i/d_i-\Mag(\ivr{b}_i))[-1,1]}
{\inum{a}_{ii}+\alpha_i[-1,1]},
\quad i=1,\dots,n.
\end{align}
\end{theorem}

%GS: Neu1990: thm4.4.10, pp. 150
%Kraw: Neu1990: (15), pp. 129
The disadvantage of the Hansen--Bliek--Rohn method is that we have to compute the inverse of $\comp{A}$. Besides this method, there are other procedures to compute a verified enclosure to $\Ss$; see \cite{MooKea2009,Neu1990}. They are usually faster, on account of tightness of the resulting enclosures. We briefly remind two of them, the well known interval Gauss--Seidel and Krawczyk iteration methods. Let $\ivr{x}\supseteq\Ss$ be in initial enclosure of $\Ss$. The Krawczyk method is based on the operator
$$
\ivr{x}\mapsto \ivr{b}+(I_n-\imace{A})\ivr{x},
$$
Denote by $\imace{D}$ the interval diagonal matrix, whose diagonal is the same as that of $\imace{A}$, and $\imace{A}'$ is used for the interval matrix $\imace{A}$ with zero diagonal. The interval Gauss--Seidel operator reads
$$
\ivr{x}\mapsto \imace{D}^{-1}(\ivr{b}-\imace{A}'\ivr{x}).
$$
In fact, this operator is often called the interval Jacobi operator, whereas the  interval Gauss--Seidel one raises by evaluating the above expression row by row and using the already tightened entries of $\ivr{x}$ in the subsequent rows. Anyway, the limit enclosures are the same, so for the sake of simplicity, we will employ this formulation.

By $\ivr{x}^{\textnormal{GS}}$ and $\ivr{x}^{\textnormal{K}}$ we denote the limit enclosures computed by the interval Gauss--Seidel and Krawczyk methods, respectively.
The theorem below adapted from \cite{Neu1990} gives an explicit formulae for the enclosures.

\begin{theorem}
We have
\begin{align*}
\ivr{x}^{\textnormal{GS}}&=\imace{D}^{-1}(\ivr{b}+\Mag(\imace{A}')u[-1,1]),\\
\ivr{x}^{\textnormal{K}}&=\ivr{b}+\Rad{A}u[-1,1].
\end{align*}
Moreover,
\begin{align}\label{eqMag}
u=\Mag(\ivr{\Ss})
=\Mag(\ivr{x}^{\textnormal{GS}})
=\Mag(\ivr{x}^{\textnormal{K}}).
\end{align}
\end{theorem}

Property \nref{eqMag}, not stressed enough in the literature, shows an interesting relation between the mentioned methods. In each coordinate, all corresponding enclosures have one endpoint in common (that one with the larger absolute value). Thus, the enclosures differ from one side only (but the difference may be large)

%%%%%%%%%%%%%%%%%%%%%%%%%%%%%%%%%%%%%%%%%%%%%%%%%%%%%%%%%%%%%%% 
% OPERATOR
%%%%%%%%%%%%%%%%%%%%%%%%%%%%%%%%%%%%%%%%%%%%%%%%%%%%%%%%%%%%%%% 
\section{New interval operator}

\begin{theorem}\label{thmOper}
Let $\Ss\subseteq\ivr{x}\in\IR^n$. Then
\begin{align}\label{inclThmOper}
\inum{\Ss}_i\subseteq 
 \frac{\ivr{b}_i-\sum_{j\not=i}\inum{a}_{ij}\inum{x}_j
 +[\gamma_i,-\gamma_i]u_i}
{\inum{a}_{ii}+\gamma_i[-1,1]}
\end{align}
for every $\gamma_i\in[0,\alpha_i]$ and $i=1,\dots,n$.
\end{theorem}

\begin{proof}
Let $i\in\seznam{n}$. First, we prove the statement for $\gamma_i=\alpha_i$. By Theorem~\ref{thmNK}, 
\begin{align*}
\inum{\Ss}_i=
% \frac{\ivr{b}_i-[-u_i/d_i+\Mag(\ivr{b}_i),u_i/d_i-\Mag(\ivr{b}_i)]}
 \frac{\ivr{b}_i+(u_i/d_i-\Mag(\ivr{b}_i))[-1,1]}
{\inum{a}_{ii}+\alpha_i[-1,1]}.
\end{align*}
The denominator is the same as in \nref{inclThmOper}, and it is a positive interval. Thus, it is sufficient to compare the numerators only. We have
\begin{align*}
\ivr{b}_i+(u_i/d_i-\Mag(\ivr{b}_i))[-1,1]
&=\ivr{b}_i+(u_i/d_i-(\comp{A}u)_i)[-1,1]\\
&=\ivr{b}_i+\left(\sum_{j\not=i}\Rad{a_{ij}}u_j
 -(\comp{a_{ii}}-1/d_i)u_i\right)[-1,1]\\
&\subseteq\ivr{b}_i+\left(\sum_{j\not=i}\Rad{a_{ij}}\Mag(\inum{x}_j)
 -\gamma_iu_i\right)[-1,1]\\
&=\ivr{b}_i-\sum_{j\not=i}\inum{a}_{ij}\inum{x}_j
 +[\gamma_i,-\gamma_i]u_i.
\end{align*}
For $\gamma_i=0$, \nref{inclThmOper} reduces to the interval Gauss-Seidel operator. 

Now, we suppose that $0<\gamma_i<\alpha_i$. Denoting $\inum{v}_i:=\ivr{b}_i-\sum_{j\not=i}\Rad{a_{ij}}u_j[-1,1]$, we have to show the inclusion
\begin{align*}
\frac{\inum{v}_i+\alpha_iu_i[1,-1]}{\inum{a}_{ii}+\alpha_i[-1,1]}
\subseteq\frac{\inum{v}_i+\gamma_iu_i[1,-1]}{\inum{a}_{ii}+\gamma_i[-1,1]}.
\end{align*}
We show it by comparing the left endpoints only; the right endpoints are compared accordingly. We distinguish three cases:

1) Let $\unum{v}_i+\gamma_iu_i\geq0$. Then we want to show that 
$$
\frac{\unum{v}_i+\gamma_iu_i}{\onum{a}_{ii}+\gamma_i}\leq
 \frac{\unum{v}_i+\alpha_iu_i}{\onum{a}_{ii}+\alpha_i}.
$$
This is simplified to 
$$
\unum{v}_i(\alpha_i-\gamma_i)\leq\onum{a}_{ii}u_i(\alpha_i-\gamma_i),
$$
or,
$$
\unum{v_i}\leq\onum{a}_{ii}u_i,
$$
which is always true.

2) Let $\unum{v}_i+\gamma_iu_i<0$ and $\unum{v}_i+\alpha_iu_i\geq0$. Then the statement is obvious.

3) Let $\unum{v}_i+\alpha_iu_i<0$. Then we want to show that 
$$
\frac{\unum{v}_i+\gamma_iu_i}{\unum{a}_{ii}-\gamma_i}\leq
 \frac{\unum{v}_i+\alpha_iu_i}{\unum{a}_{ii}-\alpha_i}.
$$
Simplifying to 
$$
-\unum{v}_i(\alpha_i-\gamma_i)\leq\unum{a}_{ii}u_i(\alpha_i-\gamma_i),
$$
or,
$$
-\unum{v}_i\leq\unum{a}_{ii}u_i,
$$
which holds true.
\end{proof}

Obviously, for $\gamma=0$ we get the interval Gauss--Seidel operator, so our operator can be viewed as its generalization.
The proof also shows that the best choice for $\gamma$ is $\gamma=\alpha$.
In order to make the operator \nref{inclThmOper} applicable, we have to compute $u$ and $d$ or their lower bounds. The tighter bounds the better, however, if we spend to much time to calculate almost exact $u$ and $d$, then it makes no sense to use the operator when we can call the Ning--Kearfott formula directly. So, it is preferable to derive cheap and possibly tight lower bounds on $u$ and $d$. We suggest the following ones:

\begin{proposition}\label{propUD}
We have
\begin{align*}
u&\geq
% \hat{u}:=
\Mag{(\ivr{b})}+\Rad{A}(\Mag{(\ivr{b})}+\Rad{A}\Mag{(\ivr{b})})),\\
d_i&\geq
 \uvr{d}_i:=
%\diag(\omace{A}+(\Rad{A})^2).
%\onum{a}_{ii}+((\Rad{A})^2)_{ii}/\unum{a}_{ii},\quad i=1,\dots,n.
\onum{a}_{ii}/(1-((\Rad{A})^2)_{ii}),\quad i=1,\dots,n.
\end{align*}
\end{proposition}

\begin{proof}
The first part follows from
\begin{align*}
u&={\comp{A}}^{-1}\Mag{(\ivr{b})}
=(I_n-\Rad{A})^{-1}\Mag{(\ivr{b})}
=\left(\sum_{k=0}^\infty(\Rad{A})^k\right)\Mag{(\ivr{b})}\\
&\geq (I_n+\Rad{A}+(\Rad{A})^2)\Mag{(\ivr{b})}
=\Mag{(\ivr{b})}+\Rad{A}(\Mag{(\ivr{b})}+\Rad{A}\Mag{(\ivr{b})})).
\end{align*}
The second part follows from
\begin{align*}
d&=\diag{({\comp{A}}^{-1})}
=\diag{\left(\sum_{k=0}^\infty(\Rad{A})^k\right)},
\end{align*}
whence
\begin{align*}
d_i
&=\sum_{k=0}^\infty((\Rad{A})^k)_{ii}
\geq \onum{a}_{ii}
 +((\Rad{A})^2)_{ii}(1+\Rad{a_{ii}}+((\Rad{A})^2)_{ii}
 +((\Rad{A})^2)_{ii}\Rad{a_{ii}}+((\Rad{A})^2)_{ii}^2+\dots)\\
&=\onum{a}_{ii}+(1+\Rad{a_{ii}})((\Rad{A})^2)_{ii}(1+((\Rad{A})^2)_{ii}+((\Rad{A})^2)_{ii}^2+\dots)\\
&=\onum{a}_{ii}+\onum{a}_{ii}((\Rad{A})^2)_{ii}
 \frac{1}{1-((\Rad{A})^2)_{ii}}
=\frac{\onum{a}_{ii}}{1-((\Rad{A})^2)_{ii}}.
\qedhere
\end{align*}
\end{proof}

Notice that both bounds require computational time $\mna{O}(n^2)$. It particular, the diagonal of $(\Rad{A})^2$ is computable in square time, but the exact diagonal of $(\Rad{A})^3$ would be too costly.
The following result shows that provided we have a tight approximation on $u$, then the above estimation of $d$ is tight enough to ensure that $\gamma\geq0$. Notice that this would not be satisfied in general if we used the simpler estimation $d\geq\diag(\omace{A}+(\Rad{A})^2)$.

\begin{proposition}
We have $\gamma_i:=\comp{a_{ii}}-1/\unum{d}_i\geq0$, $ i=1,\dots,n$.
\end{proposition}

\begin{proof}
We can write
\begin{align*}
\gamma_i
&=\comp{a_{ii}}-1/\unum{d}_i
=\comp{a_{ii}}-\frac{1-((\Rad{A})^2)_{ii}}{\onum{a}_{ii}}\\
&\geq\comp{a_{ii}}-\frac{1-(\Rad{a_{ii}})^2}{1+\Rad{a_{ii}}}
=1-\Rad{a_{ii}}-(1-\Rad{a_{ii}})=0.
\qedhere
\end{align*}

\end{proof}

\subsection{Comparison to the interval Gauss--Seidel method}\label{ssCompGS}

%\begin{proposition}
%If we know $u$ exactly, then our operator is superior to the interval Gauss--Seidel operator.
%\end{proposition}

Since our operator is a generalization of the interval Gauss--Seidel iteration, it is natural to compare them. 
Let $\ivr{x}$ be an enclosure to $\Ss$, let $i\in\seznam{n}$, and denote by $\hat{u}$ a lower bound estimation on $u$. We compare the results of ours and the interval Gauss--Seidel operators, that is,
% and without loss of generality assume that $\Mid{\Ss_i}\geq0$. Then
\begin{align*}
%\ivr{x}^*_i&=
\frac{\ivr{b}_i-\sum_{j\not=i}\inum{a}_{ij}\inum{x}_j
 +[\gamma_i,-\gamma_i]\hat{u}_i}
{\inum{a}_{ii}+\gamma_i[-1,1]}\quad \mbox{and}\quad  
%\ivr{x}^{\textnormal{GS}}_i &=
%\frac{\ivr{b}_i+(\Mag(\imace{A}')u)_i[-1,1]}{\inum{a}_{ii}}=
\frac{\ivr{b}_i-\sum_{j\not=i}\inum{a}_{ij}\inum{x}_j}{\inum{a}_{ii}}.
\end{align*}
If $\gamma_i=0$, then both intervals coincide, so let us assume that $\gamma_i>0$.
Denote $\inum{v}_i:=\ivr{b}_i-\sum_{j\not=i}\inum{a}_{ij}\inum{x}_j$. We compare the left endpoints of the intervals
\begin{align*}
\frac{\inum{v}_i +[\gamma_i,-\gamma_i]\hat{u}_i}
{\inum{a}_{ii}+\gamma_i[-1,1]}\quad \mbox{and}\quad  
\frac{\inum{v}_i}{\inum{a}_{ii}},
\end{align*}
the right endpoints are compared accordingly. We distinguish three cases:

1) Let $\unum{v}_i\geq0$. Then we want to show that 
$$
\frac{\unum{v}_i}{\onum{a}_{ii}}\leq
 \frac{\unum{v}_i+\gamma_i\hat{u}_i}{\onum{a}_{ii}+\gamma_i}.
$$
This is simplified to 
$$
\unum{v}_i\gamma_i\leq \onum{a}_{ii}\hat{u}_i\gamma_i,
$$
or,
$$
\unum{v_i}\leq \onum{a}_{ii}\hat{u}_i.
$$
If $\hat{u}_i=u_i$, or $\hat{u}_i$ is not far from $u_i$, then the inequality holds true.

2) Let $\unum{v}_i<0$ and $\unum{v}_i+\gamma_iu_i\geq0$. Then the inequality is obviously satisfied.

3) Let $\unum{v}_i+\gamma_iu_i<0$. Then we want to show that 
$$
\frac{\unum{v}_i}{\unum{a}_{ii}}\leq
 \frac{\unum{v_i}+\gamma_iu_i}{\unum{a}_{ii}-\gamma_i}.
$$
This is simplified to 
$$
-\unum{v}_i\gamma_i\leq \unum{a}_{ii}u_i\gamma_i.
$$
or,
$$
-\unum{v}_i\leq \unum{a}_{ii}\hat{u}_i.
$$
This is true provided both $\hat{u}_i$ and $\unum{v}_i$ are tight enough.

The above discussion indicates that our operator with $\gamma_i>0$ is effective only if $\ivr{x}$ is sufficiently tight and the reduction of the enclosure is valid from the smaller side (in the absolute value sense) only. Since $\inum{a}_{ij}$, $i\not=j$, are symmetric intervals, the reduction in the smaller sides of $\inum{x}_i$s makes no improvement in the next iterations. The only influence is by the size of $\Mag(\ivr{x})$ since 
$$
\sum_{j\not=i}\inum{a}_{ij}\inum{x}_j
=\sum_{j\not=i}\inum{a}_{ij}\Mag(\inum{x})_j.
$$
Therefore, the following incorporation of our operator seems the most effective: Compute  $\ivr{x}\supseteq\Ss$ by the interval Gauss--Seidel method, and then call one iteration of our operator.

\begin{example}
Let
$$
\imace{A}=\begin{pmatrix}
-[8,10]&[3,5]&[8,10]\\
{}-[5,7]&[0,2]&-[6,8]\\
{}[4,6]&[7,9]&-[5,7]
\end{pmatrix},
\quad
\ivr{b}=\begin{pmatrix}[3,5]\\{}[6,8]\\{}[5,7]\end{pmatrix},
$$
and consider the interval linear system $\imace{A}x=\ivr{b}$ preconditioned by the numerically computed inverse of $\Mid{A}$. The interval Gauss--Seidel method terminates in four iterations, yielding the enclosure
$$
\ivr{x}^1=([-1.2820,0.0174],\,[0.1847, 1.5641],\,[-1.0822, 0.0889])^T;
$$
it is not yet equal to the limit enclosure 
$$
\ivr{x}^{\textnormal{GS}}
=([-1.2813,0.0167],\,[0.1849, 1.5637],\,[-1.0821, 0.0887])^T
$$
due to the limit number of iterations. By  Proposition~\ref{propUD}, we obtain the following lower bounds
\begin{align*}
u &\geq (1.1633, 1.4367, 0.9788)^T,\\
d &\geq (1.2343, 1.2536, 1.2030)^T,
\end{align*}
whence we calculate
\begin{align*}
\gamma := (0.0387,  0.0396,  0.0366)^T.
\end{align*}
These values are quite conservative since the optimal values would be for $\gamma=\alpha$, where
\begin{align*}
\alpha = (0.0632, 0.0643, 0.0604)^T.
\end{align*}
Nevertheless, the computed $\gamma$ is sufficient to reduce the  overestimation of $\ivr{x}^1$. One iteration of our operator results in the tighter enclosure
$$
\ivr{x}^2=([-1.2820,-0.0258],\,[0.2261, 1.5641],\,[-1.0822, 0.0497])^T.
$$
For completeness, notice that the interval hull of the preconditioned system is
$$
\ivr{\Ss}=([-1.2813, -0.0549],\,[0.2571, 1.5637],\,[-1.0821, 0.0144])^T.
$$

\end{example}

%%%%%%%%%%%%%%%%%%%%%%%%%%%%%%%%%%%%%%%%%%%%%%%%%%%%%%%%%%%%%%% 
% METHOD
%%%%%%%%%%%%%%%%%%%%%%%%%%%%%%%%%%%%%%%%%%%%%%%%%%%%%%%%%%%%%%% 
\section{Magnitude method}

%Even though the presented interval operator is directly applicable as an improvement and generalization of the interval Gauss--Seidel method, we propose yet another enclosure method.
Property \nref{eqMag} and the analysis at the end of Section~\ref{ssCompGS} motivate us to compute enclosure to $\Ss$ along the following lines. First, we compute the magnitude of $\Ss$, that is, $u=\comp{A}^{-1}\Mag(\ivr{b})$, and then we apply one iteration of the presented operator on the initial box $\ivr{x}=[-u,u]$, producing
$$
\frac{\ivr{b}_i-\sum_{j\not=i}\inum{a}_{ij}u_j
 +[\gamma_i,-\gamma_i]u_i}
{\inum{a}_{ii}+\gamma_i[-1,1]},\quad i=1,\dots,n.
$$
Herein, the lower bound on $d$ is computed by Proposition~\ref{propUD}. In view of the proof of Theorem~\ref{thmOper}, we can express the result equivalently as \nref{inclNK}, but in that formula, an upper bound on $d$ is required, so we do not consider it here.
Instead, we reformulate it the slightly simpler form omitting improper intervals:
$$
\frac{\ivr{b}_i+(\sum_{j\not=i}\Rad{a_{ij}}u_j-\gamma_iu_i)[-1,1]}
{\inum{a}_{ii}+\gamma_i[-1,1]},\quad i=1,\dots,n.
$$

Algorithm~\ref{algMagMet} gives a detailed and numerically reliable description on the method.

\begin{alg}\label{algMagMet}\mbox{}
\begin{enumerate}
\item\label{step1}
Compute $\ivr{u}$, an enclosure to the solution of $\comp{A}u=\Mag(\ivr{b})$.
\item\label{step2}
Calculate $\unum{d}$, a lower bound on $d$ by Proposition~\ref{propUD}.
\item
Evaluate
\begin{align*}
\ivr{x}^*_i:=
% \frac{\ivr{b}_i+(\onum{d}_i/\unum{d}_i-\Mag(\ivr{b}_i))[-1,1]}
%{\inum{a}_{ii}+\gamma_i[-1,1]},
%\frac{\ivr{b}_i-\sum_{j\not=i}\inum{a}_{ij}u_j
% +[\gamma_i,-\gamma_i]u_i}
%{\inum{a}_{ii}+\gamma_i[-1,1]},
\frac{\ivr{b}_i
 +(\sum_{j\not=i}\Rad{a_{ij}}\onum{u}_j-\gamma_i\unum{u}_i)[-1,1]}
{\inum{a}_{ii}+\gamma_i[-1,1]},
\quad i=1,\dots,n,
\end{align*}
where $\gamma_i:=\comp{a_{ii}}-1/\unum{d}_i$.
\end{enumerate}
\end{alg}

%%%%%%%%%%%%%%%%%%%%%%%%%%%%%%%%%%%%%%%%%%%%%%%%%%%%%%%%%%%%%%% 
% PROPERTIES
%%%%%%%%%%%%%%%%%%%%%%%%%%%%%%%%%%%%%%%%%%%%%%%%%%%%%%%%%%%%%%% 
\subsection{Properties}

First, not that the computations of $u$ and $d$ in steps~\ref{step1} and~\ref{step2} are independent, so may be parallelized.

Now, let us compare the magnitude method with the Hansen--Bliek--Rohn and the interval Gauss--Seidel method. The propositions below shows, that the magnitude method is superior to the interval Gauss--Seidel method, and it gives the best possible enclosure as long as  $u$ and $d$ are determined exactly. Since $u$ is computed tightly, the possible deficiency is caused only by an underestimation of $d$.
%\subsection{Comparison to the Hansen--Bliek--Rohn method}

\begin{proposition}
%Provided $u$ and $d$ are calculated exactly, 
%our algorithm computes $\ivr{\Ss}$ exactly as well.
If $u$ and $d$ are calculated exactly, then $\ivr{x}^*=\ivr{\Ss}$.
\end{proposition}

\begin{proof}
It follows from the proof of Theorem~\ref{thmOper}.
\begin{comment}
Let $i\in\seznam{n}$ and without loss of generality assume that $\Mid{\Ss_i}\geq0$. Then
\begin{align*}
\ivr{x}^*_i&=\frac{\inum{b}_i-\sum_{j\not=i}\inum{a}_{ij}u_j+\alpha_iu_i}
 {\inum{a}_{ii}+\alpha_i[-1,1]}\cap[-u_i,u_i],\\
\ivr{\Ss}_i&=\frac{\ivr{b}_i-(u_i/d_i-\Mag(\ivr{b}_i))[-1,1]}
{\inum{a}_{ii}+\alpha_i[-1,1]}.
\end{align*}
By the assumption, $\ovr{x}^*_i=\ovr{\Ss}_i=u_i$, so we have to compare the left endpoints of $\ivr{x}^*_i$ and $\ivr{\Ss}_i$ only. Since the denominators are the same, we focus on the left endpoint of the numerator of $\ivr{x}^*_i$,
\begin{align*}
\unum{b}_i-\sum_{j\not=i}\Rad{a}_{ij}u_j+\alpha_iu_i
=\unum{b}_i-\sum_{j\not=i}\Rad{a}_{ij}u_j+\comp{a_{ii}}u_i-u_i/d_i
=\unum{b}_i+(\comp{A}u)_i-u_i/d_i,
\end{align*}
which is the same as for $\ivr{\Ss}_i$.
\end{comment}
\end{proof}

%\subsection{Comparison to the interval Gauss--Seidel method}

\begin{proposition}
We have $\ivr{x}^*\subseteq\ivr{x}^{\textnormal{GS}}$. If $\gamma=0$, then there is equality.
\end{proposition}

\begin{proof}
Let $i\in\seznam{n}$ and without loss of generality assume that $\Mid{\Ss_i}\geq0$. Then
\begin{align*}
\ivr{x}^*_i&=
 \frac{\inum{b}_i-\sum_{j\not=i}\inum{a}_{ij}u_j+[\gamma_i,-\gamma_i]u_i}
 {\inum{a}_{ii}+\gamma_i[-1,1]},\\
\ivr{x}^{\textnormal{GS}}_i
 &=\frac{\ivr{b}_i-(\imace{A}'[-u,u])_i}{\inum{a}_{ii}}
 =\frac{\ivr{b}_i-\sum_{j\not=i}\inum{a}_{ij}u_j}{\inum{a}_{ii}}.
\end{align*}
Denoting $\inum{v}_i:=\ivr{b}_i-\sum_{j\not=i}\inum{a}_{ij}u_j$, we can rewrite it as
\begin{align*}
\ivr{x}^*_i&=\frac{\inum{v_i}+[\gamma_i,-\gamma_i]u_i}
 {\inum{a}_{ii}+\gamma_i[-1,1]},\\
\ivr{x}^{\textnormal{GS}}_i
 & =\frac{\inum{v}_i}{\inum{a}_{ii}}.
\end{align*}

By the assumption, $\ovr{x}^*_i=\ovr{x}^{\textnormal{GS}}_i=u_i$, so we have to compare the left endpoints of $\ivr{x}^*_i$ and $\ivr{x}^{\textnormal{GS}}$ only. We distinguish three cases:

1) Let $\unum{v}_i\geq0$. Then we want to show that 
$$
\frac{\unum{v}_i}{\onum{a}_{ii}}\leq
 \frac{\unum{v}_i+\gamma_iu_i}{\onum{a}_{ii}+\gamma_i}.
$$
This is simplified to 
$$
\unum{v}_i\gamma_i\leq \onum{a}_{ii}u_i\gamma_i.
$$
If $\gamma_i=0$, then it holds as equation, otherwise for any $\gamma_i>0$ it is true as well.

2) Let $\unum{v}_i<0$ and $\unum{v}_i+\gamma_iu_i\geq0$. Then the statement is obvious.

3) Let $\unum{v}_i+\gamma_iu_i<0$. Then we want to show that 
$$
\frac{\unum{v}_i}{\unum{a}_{ii}}\leq
 \frac{\unum{v}_i+\gamma_iu_i}{\unum{a}_{ii}-\gamma_i}.
$$
This is simplified to 
$$
-\unum{v}_i\gamma_i\leq \unum{a}_{ii}u_i\gamma_i.
$$
This is true for any $\gamma_i\geq0$, too.
\end{proof}

%%%%%%%%%%%%%%%%%%%%%%%%%%%%%%%%%%%%%%%%%%%%%%%%%%%%%%%%%%%%%%% 
% EXAMPLES
%%%%%%%%%%%%%%%%%%%%%%%%%%%%%%%%%%%%%%%%%%%%%%%%%%%%%%%%%%%%%%% 
\subsection{Numerical examples}

\begin{example}\label{ex2D}
Consider the interval linear system $\imace{A}x=\ivr{b}$, with
$$
\imace{A}=\begin{pmatrix}-[2,4]&[8,10]\\{}[2,4]&[4,6]\end{pmatrix},
\quad
\ivr{b}=\begin{pmatrix}-[4,6]\\{}-[8,10]\end{pmatrix}.
$$
\begin{figure}[t]\label{fig2D}
\begin{center}
\psset{unit=2.4cm,arrowscale=1.5}
%\psset{xunit=.24cm,yunit=.26cm,arrowscale=1.5}
\begin{pspicture}(-4.1,-2.2)(1.4,0.5)
\newgray{mygray}{0.85}
\pspolygon[fillstyle=solid,fillcolor=mygray,linewidth=1pt]
(-0.5806,-0.6774)(-0.4,-1.4)(-3.4545, -1.9091)(-2.7059, -0.4118)(-0.5806,-0.6774)
\pspolygon[fillstyle=solid,fillcolor=lightgray,linewidth=0.5pt]
(-3,-1)(-1.75,-1.625)(-0.5,-1)(-1.0769,-0.6154)(-3,-1)
%verifylss
\psline[linewidth=1pt](-3.4985, -1.9279)
(-3.4985, -0.0721)(0.8318,-0.0721)(0.8318,-1.9279)(-3.4985, -1.9279)
%GS
\psline[linewidth=1pt](-3.4555, -1.9093)
(-3.4555,-0.3180)(-0.2722,-0.3180)(-0.2722,-1.9093)(-3.4555,-1.9093)
%magnitude method
\psline[linewidth=1pt](-3.4546, -1.9091)
(-3.4546,-0.3741)(-0.3557,-0.3741)(-0.3557,-1.9091)(-3.4546,-1.9091)
%magnitude method GS
\psline[linewidth=1pt](-3.4546, -1.9091)
(-3.4546,-0.3181)(-0.2727,-0.3181)(-0.2727,-1.9091)(-3.4546,-1.9091)
\psaxes[ticksize=2pt,labels=all,ticks=all, linewidth=0.5pt]{->}(0,0)(-4,-2.2)(1.3,0.4)
\uput[-135](-0.1,-0.1){$0$}
\uput[-90](1.3,0){$x_1$}
\uput[-180](0,0.4){$x_2$}
\end{pspicture}
\end{center}
\caption{(Example~\ref{ex2D}) The solution set in gray, the preconditioned one in light gray, and three enclosures for \texttt{verifylss}, the interval Gauss--Seidel and the magnitude method (from the largest).}
\end{figure}
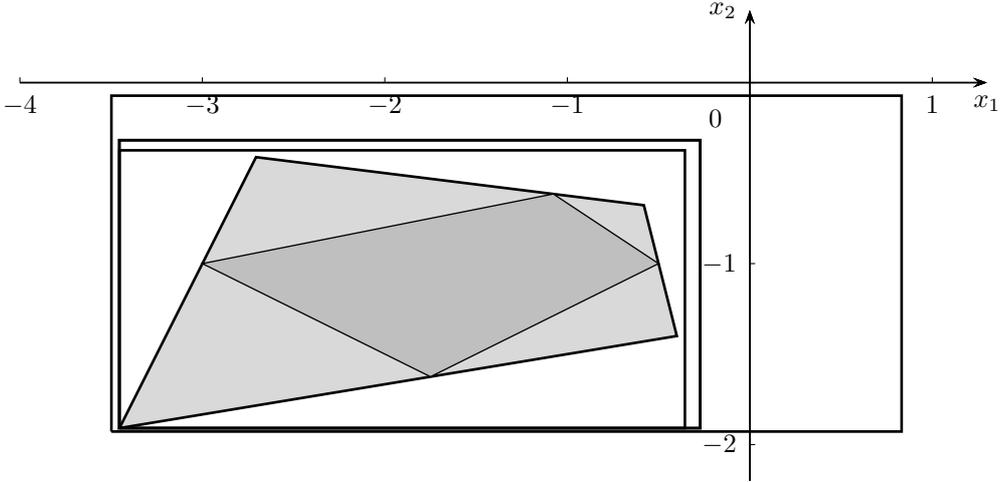
Figure~\ref{fig2D} depicts the solution set to $\imace{A}x=\ivr{b}$ in gray color, and the preconditioned system by $(\Mid{A})^{-1}$ in light gray. We compare three methods for enclosing the solution set. The function \texttt{verifylss} from the package  \textsf{INTLAB} \cite{Rum1999} yields the enclosure
$$
\ivr{x}^1=([ -3.4985,  0.8318],\,[ -1.9279, -0.0721])^T.
$$
The interval Gauss--Seidel method gives tighter enclosure
$$
\ivr{x}^2=([-3.4555,-0.2722],\,[-1.9093,-0.3180] )^T,
$$
but it requires almost double computational time. In contrast, our magnitude method produces yet a bit tighter enclosure
$$
\ivr{x}^*=([ -3.4546, -0.3557],\,[-1.9091, -0.3741])^T,
$$
but with less computational effort than the other methods. The enclosure is also very close to the optimal one (for the preconditioned system)
$$
\ivr{\Ss}=([-3.4546,-0.3999],\,[-1.9091, -0.4117])^T.
$$
Enclosures $\ivr{x}^1$, $\ivr{x}^2$, $\ivr{x}^*$ are illustrated in Figure~\ref{fig2D} respectively in a nested way.
\end{example}

In the example below, we present a limited computational study.

\begin{example}\label{exNumer}
We considered randomly generated examples for various dimensions and interval radii. The entries of $\Mid{A}$ and $\Mid{b}$ were generated randomly in $[-10,10]$ with uniform distribution. All radii of $\imace{A}$ were equal to the parameter $\delta>0$.

%The computations  were carried out in \textsf{MATLAB 7.11.0.584 (R2010b)} on a machine with \texttt{AMD Athlon 64 X2 Dual Core Processor 4400+}, CPU 2.2~GHz, with 1004 MB RAM. Interval arithmetics and some basic interval functions were provided by the interval toolbox \textsf{INTLAB v6} \cite{Rum1999}. 
The computations  were carried out in \textsf{MATLAB 7.11.0.584 (R2010b)} on a six-processor machine \texttt{AMD Phenom(tm) II X6 1090T Processor}, CPU 800~MHz, with 15579\,MB RAM. Interval arithmetics and some basic interval functions were provided by the interval toolbox \textsf{INTLAB v6} \cite{Rum1999}. 

We compared four methods with respect to computational time and tightness of resulting enclosures. Namely, \texttt{verifylss} function from the \textsf{INTLAB}, the interval Gauss--Seidel method, the proposed magnitude method (Algorithm~\ref{algMagMet}), and eventually the magnitude method with $\gamma=0$. The last one yields the limit Gauss--Seidel enclosure, and it is faster than the magnitude method since we need not compute a lower bound on $d$.

Table~\ref{tabNumerTime} shows the running times in seconds, and Table~\ref{tabNumerTight} shows the tightness for the same data. The tightness was  measured by the sum of the resulting interval radii with respect to the optimal interval hull $\ivr{\Ss}$ computed by the Ning--Kearfott formula \nref{inclNK}. Precisely, we display
\begin{align*}
\frac{\sum_{i=1}^n\Rad{x_i}}{\sum_{i=1}^n\Rad{\Ss_i}},
\end{align*}
where $\ivr{x}$ is the calculated enclosure. Thus, the closer to 1, the sharper enclosure.

\begin{table}[t]%\small%\footnotesize
\caption{(Example~\ref{exNumer}) Computational time for randomly generated data.\label{tabNumerTime}}
\begin{center}
\tabcolsep=5pt
%\begin{sideways}
%\begin{tabular}{@{}rllll@{}}
\begin{tabular}{rlllll}
 \toprule
$n$ & $\delta$ & \texttt{verifylss} &  \texttt{Gauss-Seidel} &  \texttt{magnitude}&  \texttt{magnitude ($\gamma=0$)}\\
\midrule 
  5 & 1    & 3.2903   & 0.10987 & 0.004466 & 0.003429\\
  5 & 0.1  & 0.004234 & 0.02937 & 0.004513 & 0.003502\\
  5 & 0.01 & 0.002342 & 0.02500 & 0.004473 & 0.003456\\
 10 & 0.1  & 0.018845 & 0.08370 & 0.004877 & 0.003777\\
 10 & 0.01 & 0.003161 & 0.05305 & 0.004821 & 0.003799\\
 15 & 0.1  & 0.246779 & 0.21868 & 0.005212 & 0.004162\\
 15 & 0.01 & 0.005403 & 0.09163 & 0.005260 & 0.004172\\
 20 & 0.1  & 16.9678  & 0.95238 & 0.005554 & 0.004251\\
 20 & 0.01 & 0.008950 & 0.15602 & 0.005736 & 0.004622\\
 30 & 0.01 & 0.019111 & 0.32294 & 0.006457 & 0.005289\\
 30 & 0.001& 0.004488 & 0.19544 & 0.006460 & 0.005260\\
 50 & 0.01 & 0.210430 & 1.01155 & 0.008483 & 0.007062\\
 50 & 0.001& 0.010190 & 0.54813 & 0.008343 & 0.006879\\
100 & 0.001& 0.044463 & 2.42025 & 0.016706 & 0.014645\\
100 &0.0001& 0.013940 & 1.48693 & 0.017089 & 0.014847\\
%100& 0.01 &  &  &  \\
\bottomrule
\end{tabular}
%\end{sideways}
\end{center}
\end{table}

\begin{table}[t]%\small%\footnotesize
\caption{(Example~\ref{exNumer}) Tightness of enclosures for randomly generated data.\label{tabNumerTight}}
\begin{center}
\tabcolsep=5pt
%\begin{sideways}
%\begin{tabular}{@{}rllll@{}}
\begin{tabular}{rlllll}
 \toprule
$n$ & $\delta$ & \texttt{verifylss} &  \texttt{Gauss-Seidel} &  \texttt{magnitude}&  \texttt{magnitude ($\gamma=0$)}\\
\midrule 
  5 & 1    & 1.1520  & 1.1510  & 1.09548  & 1.1196\\
  5 & 0.1  & 1.08302 & 1.01645 & 1.00591  & 1.0164\\
  5 & 0.01 & 1.01755 & 1.00148 & 1.00037  & 1.00148\\
 10 & 0.1  & 1.07756 & 1.02495 & 1.01107  & 1.02474\\
 10 & 0.01 & 1.02362 & 1.00378 & 1.00132  & 1.00378\\
 15 & 0.1  & 1.06994 & 1.03121 & 1.01755  & 1.03074\\
 15 & 0.01 & 1.02125 & 1.00217 & 1.00047  & 1.00216\\
 20 & 0.1  & 1.05524 & 1.03076 & 1.02007  & 1.02989\\
 20 & 0.01 & 1.02643 & 1.00348 & 1.00097  & 1.00348\\
 30 & 0.01 & 1.02539 & 1.00402 & 1.00129  & 1.00401 \\
 30 & 0.001& 1.00574 & 1.00026 & 1.000039 & 1.000256\\
 50 & 0.01 & 1.02688 & 1.00533 & 1.00226  & 1.00531\\
 50 & 0.001& 1.00902 & 1.00051 & 1.00011  & 1.00051\\
100 & 0.001& 1.01303 & 1.00057 & 1.00013  & 1.00057\\
100 &0.0001& 1.0024988 & 1.0000274 & 1.0000022 & 1.0000274\\
%100& 0.01 &  &  &  \\
\bottomrule
\end{tabular}
%\end{sideways}
\end{center}
\end{table}

The results of our experiments show that the magnitude method with $\gamma=0$ saves some time (about $10\%$ to $20\%$), but the loss in tightness may be larger. Compared to the interval Gauss--Seidel method, the magnitude method wins significantly both in time and tightness. Compared to \texttt{verifylss}, our approach produces tighter enclosures. Provided interval entries of the equation system are wide, the magnitude method is also cheaper; for narrow enough intervals, the situation is changed and \texttt{verifylss} needs less computational effort.

For both variants of the magnitude method, we used \texttt{verifylss} for computing a verified enclosure  of $u=\comp{A}^{-1}\Mag(\ivr{b})$ (step~\ref{step1} of Algorithm~\ref{algMagMet}). So it might seem curious that (for wide input intervals) \texttt{verifylss} beats itself.
\end{example}

%%%%%%%%%%
\section{Conclusion}

We proposed a new operator for tightening solution set enclosures of interval linear equations. Based on this operator and a property of limit enclosures of classical methods, we came up with a new algorithm, called the magnitude method. It always outperforms the interval Gauss--Seidel method. Numerical experiments indicate that it is efficient in both computational time and tightness of enclosures, in particular for wide interval entries. 

In the future research, we would like to extend our approach to parametric interval systems. Also, overcoming the assumption $\Mid{A}=I_n$ and considering non-preconditioned systems is a challenging problem.
Very recently, a new version of \textsf{INTLAB} was released (unfortunately, no longer free of charge), so numerical studies utilizing enhanced \textsf{INTLAB} functions would be of interest, too.

\subsubsection*{Acknowledgments.} 

The author was supported by the Czech Science Foundation Grant P402/13-10660S.

%%%%%%%%%%%%%%%%%%%%%%%%%%%%%%%%%%%%%%%%%%%%%%%%%%%%%%%%%%%%%%% 
% REFERENCES
%%%%%%%%%%%%%%%%%%%%%%%%%%%%%%%%%%%%%%%%%%%%%%%%%%%%%%%%%%%%%%% 

\bibliographystyle{abbrv}
\bibliography{iter_hbr}

\end{document}